\documentclass[11pt]{amsart}
\usepackage[dvips]{graphicx}
\usepackage{amsmath,graphics}
\usepackage{amsfonts,amssymb,color}
\usepackage{xypic}
\theoremstyle{plain}
\newtheorem{theorem*}{Theorem}
\newtheorem*{lemma*} {Lemma}
\newtheorem{corollary*}[theorem*]{Corollary}
\newtheorem{proposition*}[theorem*]{Proposition}
\newtheorem*{conjecture*}{Conjecture}
\newtheorem{theorem}{Theorem}[section]
\newtheorem{lemma}[theorem]{Lemma}
\newtheorem*{theorem1*}{Theorem 1}
\newtheorem*{theorem2*}{Theorem 2}
\newtheorem*{theorem3*}{Theorem 3}
\newtheorem*{theorem4*}{Theorem 4}
\newtheorem*{theorem5*}{Theorem 5}
\newtheorem*{theorem6*}{Theorem 6}
\newtheorem*{corollary1*}{Corollary 1}
\newtheorem*{corollary2*}{Corollary 2}
\newtheorem*{corollary3*}{Corollary 3}
\newtheorem*{corollary4*}{Corollary 4}
\newtheorem*{corollary5*}{Corollary 5}
\newtheorem*{corollary6*}{Corollary 6}

\newtheorem{proposition}[theorem]{Proposition}
\newtheorem{conjecture}[theorem]{Conjecture}

\theoremstyle{remark}
\newtheorem*{remark}{Remark}

\newtheorem{example*}{Example}
\newtheorem*{claim}{Claim}

\theoremstyle{definition}

\def\fol{\mathcal{F}}

 \def\Q{\Bbb{Q}}  \def\Z{\Bbb{Z}} \def\R{\Bbb{R}} 
\def\N{\Bbb{N}}    
 \def\a{\alpha} \def\g{\gamma}  \def\bp{\begin{pmatrix}}
\def\sm{\setminus} \def\ep{\end{pmatrix}} \def\bn{\begin{enumerate}} 
   \def\en{\end{enumerate}}
\def\ba{\begin{array}} \def\ea{\end{array}}  
   \def\a{\alpha} \def\b{\beta}

\def\ker{\mbox{Ker}}\def\be{\begin{equation}} \def\ee{\end{equation}} 
   
 \def\hom{\mbox{Hom}}  
   \def\eps{\epsilon}

\def\w{\omega}   
    \def\fr12{\frac{1}{2}} \def\z12{\Z[\fr12]}

\begin{document}
\title[Symplectic structures on $4$--manifolds with a free circle action]{Construction of symplectic structures on $4$--manifolds with a free circle action}
\author{Stefan Friedl}
\address{Mathematisches Institut, Universit\"at zu K\"oln, Germany}
\email{sfriedl@gmail.com}
\author{Stefano Vidussi}
\address{Department of Mathematics, University of California,
Riverside, CA 92521, USA} \email{svidussi@math.ucr.edu} \thanks{The second author was partially supported by NSF grant \#0906281.}
\date{\today}

\begin{abstract}
Let $M$ be a closed 4-manifold with a free circle action.
If the orbit manifold $N^3$ satisfies an appropriate fibering condition, then
we show how to represent a cone in $H^2(M;\R)$ by symplectic forms. This generalizes earlier constructions by Thurston, Bouyakoub and Fern\'andez-Gray-Morgan.
In the case that $M$ is the product 4-manifold $S^1\times N$ our construction complements the results of \cite{FV08} and allows us to completely determine the symplectic cone of such 4-manifolds.
\end{abstract}

\maketitle

\section{Introduction and main results}

Let $M$ be a closed $4$--manifold with a free circle action. We denote the orbit space by $N$
and we denote by $p:M\to N$ the quotient map, which defines a principal $S^1$--bundle over $N$.
We denote by $p_*:H^2(M;\R)\to H^1(N;\R)$ the map given by integration along the fiber.
Our main result (which will be proved in Section \ref{sec:const}) is the following existence theorem.

\begin{theorem} \label{thm:const}
 Let $M$ be a closed, oriented $4$--manifold
admitting a free circle action.
Let $\psi \in H^2(M;\R)$ such that $\psi^2 > 0\in H^4(M;\R)$ and such that $p_*(\psi)\in H^1(N;\R)$ can be represented by a non--degenerate closed 1--form.  Then there exists an
$S^1$--invariant symplectic form $\w$ on $M$ with $[\w]=\psi \in H^2(M;\R)$.
\end{theorem}

\begin{remark}
\bn
\item Note that given $\phi \in H^1(N;\R)$ we can represent $\phi$ by a non-degenerate (i.e. nowhere zero) closed 1--form if and only if
$\phi$ lies in the cone on a fibered face of the Thurston norm ball (see \cite{Th86} for details). Therefore the theorem assumes implicitly that $N$ admits a fibration over $S^1$.
\item
This theorem generalizes work by Thurston \cite{Th76}, Bouyakoub \cite{Bou88} and Fern\'andez-Gray-Morgan \cite{FGM91}.
More precisely, Thurston first constructed symplectic forms on product manifolds $S^1\times N$ for fibered 3-manifolds $N$.
Bouyakoub generalized Thurston's results and showed that given  $\psi$ as in the theorem, there exists an $S^1$--invariant symplectic form
$\w$ with $p_*([\w])=p_*(\psi)$. Finally Fern\'andez-Gray-Morgan proved the theorem in the case that
 $p_*(\psi)$ is rational.
 \en
\end{remark}

Let $W$ be a $4$--manifold. The set of all elements of $H^2(W;\R)$ which can be represented by a symplectic form is called the \emph{symplectic cone} of $W$. Note that this is indeed a cone, i.e. if $\psi$ can be represented by a symplectic form, then any non-zero scalar multiple can also be represented by a symplectic form. Determining the symplectic cone of 4-manifolds is a fundamental problem, but little seems to be known in general.
We refer to \cite[Section~3]{Li08} for more information.

In \cite{FV08} the authors showed that if $N$ is a closed 3-manifold, then $S^1\times N$ is symplectic if and only if $N$ fibers over $S^1$.
(In the case that $b_1(N)=1$ this also follows from combining the work of Kutluhan-Taubes \cite{KT09} with the work of  Kronheimer-Mrowka \cite{KM08} and Ni \cite{Ni08}). In fact, a slightly more precise version of this result (see \cite[Theorem~1.3]{FV08}) will allows us in
Section \ref{section:3} to
determine the symplectic cone of closed 4-manifolds of the form $S^1\times N$:

\begin{theorem}\label{thm:symplecticcone}
Let $N$ be a closed, oriented 3--manifold. Then given $\psi \in H^2(S^1\times N;\R)$ the following are equivalent:
\bn
\item $\psi$ can be represented by a symplectic structure,
\item $\psi$ can be represented by a symplectic structure which is $S^1$--invariant,
\item $\psi^2>0$ and the K\"unneth component $\phi=p_*(\psi) \in H^1(N;\R)$ of $\psi$ lies in the open cone on a fibered face of the Thurston norm ball of $N$.
\en
\end{theorem}

\begin{remark}
\bn
\item
Note that we are not claiming that \textit{any} symplectic form is isotopic, or even homotopic to an $S^1$--invariant form, although this might be the case.
\item
We expect a very similar theorem to hold for closed 4-manifold with a free circle action.
In fact the proof of Theorem \ref{thm:symplecticcone} together with work of Bowden \cite{Bow09}
and the authors \cite{FV07} shows that an analogous statement holds for circle bundles $M\to N$ whenever $N$ has vanishing Thurston norm
or $N$ is a graph manifold.
\en
\end{remark}

\noindent \textbf{Convention.} All maps are assumed to be $C^\infty$ unless stated otherwise. All manifolds are assumed to be connected, compact, closed and orientable. All homology and cohomology groups are with integral coefficients, unless it says specifically otherwise.
\\

\noindent \textbf{Note.}
This paper is one section of an unpublished 2007 preprint by the authors called `Symplectic 4-manifolds with a free circle action'.
\\

\noindent \textbf{Acknowledgment.} We would like to thank Paolo Ghiggini and Ko Honda for helpful conversations.

\section{Construction of symplectic forms} \label{sec:const}

\subsection{Outline of the proof of Theorem \ref{thm:const}}

In this section we will give a proof of Theorem \ref{thm:const}  modulo some technical lemmas which will be proved in
Sections \ref{section:aux1}, \ref{section:aux2} and \ref{section:aux3}.

 For the remainder of this section let $M$ be an oriented $4$--manifold
admitting a free $S^1$-action.
We denote the orbit space by $N$
and we denote by $p:M\to N$ the quotient map, which defines a principal $S^1$--bundle over $N$.

In the following we denote by $e\in H^2(N)$ the Euler class of the $S^1$--bundle $M\to N$.
(Note that $M$ decomposes as a product $M=S^1\times N$ if and only if $e=0$.)
Recall the
Gysin sequence \be \label{sequ0}  \Z=H^0(N;\R) \xrightarrow{\cdot e}
H^2(N;\R)\xrightarrow{p^*} H^2(M;\R)\xrightarrow{p_*} H^1(N;\R) \xrightarrow{\cup e}H^3(N;\R)=\R.\ee
Here $p_{*}:H^2(M;\R)\to H^1(N;\R)$ is the map given by integration along the fiber. The same sequence can be considered for cohomology with integral coefficients.
Note that the map $p_*:H^4(M)\to H^3(N)$ is an isomorphism, and we endow $N$ with the orientation given by the image of the orientation of $M$ under $p_*$.

Throughout this section we assume that $\psi \in H^2(M;\R)$ is such that $\psi^2 > 0\in H^4(M;\R)$ and such that $p_*(\psi)\in H^1(N;\R)$ can be represented by a non--degenerate closed 1--form $\a$.

\begin{lemma} \label{lem:beta}
There exists a 1--form $\b$ on $N$ such that $\a\wedge \b$ is closed and $[\b\wedge \a]=e\in
H^2(N;\R)$.
\end{lemma}

In the case that $p_*(\psi)$ is integral this lemma is stated in \cite[Lemma~15]{FGM91}.
We give the proof of Lemma \ref{lem:beta} in Section \ref{section:beta}.

 Now let $\g=\b\wedge
\a$. Since $[\g]=e\in H^2(N;\R)$ we can find a  1--form $\eta$ (namely a connection 1--form for $M\to N$)
on $M$
with the following properties: \bn
\item $\eta$ is invariant under the $S^1$--action,
\item the integral of $\eta$ over a fiber (which inherits an orientation from $S^1$) equals 1, and
\item $d\eta=p^*(\g)$.
\en
This form is often referred to as \emph{global angular form}.
We  refer to \cite{BT82} for more details.
 Note that (1) and (2) imply that $\eta$ is non--trivial on any
non--trivial vector tangent to a fiber.

Note that $d(p^*(\a)\wedge \eta)=p^*(\a\wedge \gamma)=p^*(\a\wedge \a\wedge \b)=0$. We can therefore consider
$\psi-[p^*(\a)\wedge \eta]\in H^2(M;\R)$. It follows easily from $p_*(\psi)=[\a]$ and the second property of $\eta$ that
$p_*\big(\psi-[p^*(\a)\wedge \eta]\big)=0\in H^1(N;\R)$. By the exact sequence (\ref{sequ0}) we can therefore find
$h\in H^2(N;\R)$ with $p^*(h)=\psi-[p^*(\a)\wedge \eta]$. By assumption we have  $\psi^2 > 0$.
Note that
\[ \ba{rcl}  \psi^2&=&(p^*(h)+[p^*(\a)\wedge \eta])\cup (p^*(h)+[p^*(\a)\wedge \eta])\\
&=&p^*(h^2)+[p^*(\a)\wedge \eta]\cup [p^*(\a)\wedge \eta]+2p^*(h)\cup [p^*(\a)\wedge \eta]\\
&=&p^*(h^2)+[p^*(\a)\wedge \eta\wedge p^*(\a)\wedge \eta]+2p^*(h)\cup [p^*(\a)\wedge \eta]
.\ea \]
The first term is zero since $N$ supports no 4-forms and the second term is zero since $\eta$ and $p^*(\a)$ are one-forms. It follows that
\[ p^*(h)\cup [p^*(\a)\wedge \eta]=\frac{1}{2}\psi^2>0\in H^4(M;\R).\]
 Recall that
the map $p_*:H^4(M;\R)\to H^3(N;\R)$ is an orientation preserving  isomorphism.
In particular  we therefore get that \[ h\cup [\a]
=p_*(p^*(h)\cup [p^*(\a)\wedge \eta]) > 0  \in H^3(N;\R).\]
We will prove the following lemma in Section \ref{section:nonzero}.

\begin{lemma} \label{lem:nonzero}
Given $h\in H^2(N;\R)$ with $h\cup [\a]> 0$ we can find a representative $\Omega$ of $h$ such that
$\Omega\wedge \a> 0$ everywhere.
\end{lemma}

It is now clear that the following claim concludes the proof of Theorem \ref{thm:const}.
\begin{claim}
\[ \w=p^*(\Omega)+p^*(\a)\wedge \eta \]
is an $S^1$--invariant symplectic form on $M$ which represents $\psi$.
\end{claim}

It is clear that $\w$ is $S^1$--invariant. We compute
\[ d\w=d(p^*(\Omega)+p^*(\a)\wedge \eta)=p^*(\a)\wedge d\eta=p^*(\a\wedge \g)=p^*(\a\wedge \a\wedge \b)=0,\]
i.e. $\w$ is closed.
 Also note that
\[ \psi=p^*(h)+[p^*(\a)\wedge \eta]=[p^*(\Omega)+p^*(\a)\wedge \eta],\]
i.e. $\w$ represents $\psi$.
It remains to show that $\w\wedge \w$ is positive
everywhere. For any point $q \in M$ pick a
basis $a,b,c,d$ for the tangent space $T_{q}M$ such that: \bn
\item $p_*(a),p_*(b)$ are a basis for the
tangent space $\mbox{ker} \a|_{p(q)}$ of a leaf of the foliation on $N$ determined by $\a$; put differently, $\a(p_*(a))=\a(p_*(b))=0$ and $p_*(a), p_*(b)$ are linearly independent,
\item $\a(p_*(c)) >  0$,
\item $d$ is tangent to the fibers of the $S^1$--fibration $M\to N$ and $\eta(d) > 0$.
\en Note that $p_*(d)=0$ and $p^*(\a)$ vanishes on $a,b,d$. It is now easy to see that
\[ \ba{rcl}(\w\wedge \w)(a,b,c,d)
&=&2(p^*(\Omega) \wedge p^*(\a)\wedge \eta)(a,b,c,d)\\
&=&2p^*(\Omega)(a,b)\cdot p^*(\a)(c)\cdot \eta(d)
=2\Omega(p_*(a),p_*(b))\cdot \a(p_*(c))\cdot \eta(d)\\
&=&2(\Omega\wedge \alpha)(p_*(a),p_*(b),p_*(c))\cdot \eta(d).\ea\]
Since $\Omega\wedge \alpha$ is a non-zero 3-form and since $p_*(a),p_*(b),p_*(c)$ form a basis for the tangent space of $N$ we see that
the last expression is in fact non-zero. This shows that $\w\wedge \w$ is non-zero everywhere, but since $\w \wedge \w$ represents the positive class $\psi^2$ we see that $\w \wedge \w$ is in fact positive throughout.
 This
concludes the proof of the claim and hence the proof of Theorem \ref{thm:const}

\subsection{Non--degenerate closed 1--forms and dual curves} \label{section:aux1}

Throughout this section $\a$ will be a non--degenerate closed 1--form on $N$. Note that $\a$
(or strictly speaking $\ker(\a)$)
defines a foliation which we denote by $\fol$.
Before we can prove Lemmas \ref{lem:beta} and \ref{lem:nonzero} we need a preliminary result regarding
representability of homology classes in $N$ by smooth embedded curves transverse to, or contained in a leaf of, the foliation $\fol$. The following Lemma is presumably known (its existence is discussed e.g. in \cite{Lt87}) but we include a proof for completeness.

\begin{lemma} \label{lem1}
Let $\a$ be a non--degenerate closed 1--form on $N$ with corresponding foliation $\fol$ and let $p\in N$.
For every $h\in H^2(N;\Z)$ with $h\cup [\a]\neq 0$ (respectively, $h\cup [\a]=0$) there exists a smoothly embedded closed (possibly disconnected) curve $c$ with
$PD([c])=h$ transverse to (respectively, contained in a leaf of) the foliation $\fol$ and that goes through $p$.
\end{lemma}

\begin{proof}
Let $\a$ be a non--degenerate closed 1--form on $N$ with corresponding foliation $\fol$.
We first pick a metric $g$ on $N$. We let $v'$ be the unique vector field on $N$ with the property that for any $p\in N$ and any $w\in T_pN$ we have
$g(v'(p),w)=\a(w)$. Note that this implies that $\a(v'(p))\ne 0$ for all $p$. We then define a new vector field $v$ by
\[ v(p)=\frac{v'(p)}{\a(v'(p))}.\]
Note that $\a(v(p))=1$ for all $p\in N$.
We denote by $F:N\times \R\to N$ the flow corresponding to $-v$, i.e. for any $p\in N, \,\ s\in \R$ we have
\begin{equation} \label{equ:flow} \frac{\partial}{\partial t} F(p,t)|_{t=s}=-v(F(p,s))\end{equation} with initial condition $F(p,0) = p$ (as $N$ is compact, the flow is defined for all $s \in \R$). Observe that Cartan's formula implies that $L_{v}\a = d( i_v \a )+ i_v (d \a) = d(1) = 0$. It follows that \[ \frac{d}{ds} (F_{s}^{*} \a) = 0,\] where $F_{s} : N \to N$ is the map defined by $F_{s}(q) = F(q,s)$, hence \[ F_{s}^{*} \a = F_{0}^{*} \a  = \a.\]

We will repeatedly make use of the following formula: Given a path $(\gamma,\rho): [0,1] \to N \times \R$, by the chain rule the induced path $\eta := F(\gamma,\rho): \R \to N$ has tangent vector \[ \frac{d \eta}{dt} = \frac{d}{dt} F(\gamma(t),\rho(t)) = (F_{\rho(t)})_{*} (\frac{d \gamma}{dt}) + \frac{\partial}{\partial s} F(\gamma(t),s)|_{s = \rho(t)} \frac{d\rho}{dt}  \]
 and as usual the derivatives at the endpoints are interpreted  as one-sided. Using Equation (\ref{equ:flow}) we can rewrite this vector as \begin{equation} \label{equ:vector} \frac{d {\eta}}{dt}  = (F_{\rho(t)})_{*} \frac{d \gamma}{dt} - v (\eta(t)) \frac{d \rho}{dt} \in T_{\eta(t)}N. \end{equation}

Let now $\gamma:[0,1]\to N$ be any smoothly embedded loop with $\gamma(0)=\gamma(1)=p$ whose image
(which by abuse of notation  we will also denote by $\gamma$), is dual to a class $h \in H^2(N;\Z)$ such that $h \cup [\a] = \int_{\gamma}\a = m \in \R$. Let $\rho(t) = mt$ and denote, as above, $\eta(t) = F(\gamma(t),mt)$. Define a map $\Phi : [0,1] \to \R$ as $\Phi(t) = \int_{\eta|_{[0,t]}}\a$, where $\eta|_{[0,t]}$ denotes the restriction of the map $\eta:[0,1]\to N$ to the interval $[0,t]$.
Note that by Equation (\ref{equ:vector}) \[ \frac{d \Phi}{dt} = \a(\frac{d {\eta}}{dt}) = \a \big((F_{mt})_{*} \frac{d \gamma}{dt} - mv\big). \] Using the identities $\a((F_{s})_{*}) = F_{s}^{*} \a = \a$ and $\a(v) = 1$ we get therefore
\[ \frac{d \Phi}{dt}  = \a \big(\frac{d \gamma}{dt}\big) - m.\] In particular it follows that
 \[\Phi(1) = \int_{\eta} \a = \int_{0}^{1} \a\big(\frac{d {\eta}}{dt}\big)dt = \int_{0}^{1}
 \a ((F_{mt})_{*} \frac{d \gamma}{dt} - mv)dt = \int_{0}^{1} \a\big(\frac{d \gamma}{dt}\big)dt - m = \int_{\gamma}\a - m = 0.\]

We consider now the following homotopy
\[  \ba{rcl}
H:[0,1]\times [0,1]&\to&N \\
(t,s)&\mapsto&F(\gamma(t),s\Phi(t)).\ea \]
This is clearly a smooth map.
Since $\Phi(1)=0$ this descends in fact to a homotopy $H:S^1\times [0,1]\to N$.
Note that $H(t,0)=\gamma(t)$ for all $t$.
We now consider the path ${\tilde \gamma}(t)$ defined by ${\tilde \gamma}(t)=H(t,1)$. Note that ${\tilde \gamma}(0)={\tilde \gamma}(1)=p$. The map ${\tilde \gamma}(t)$ is smooth, and we claim that the image ${\tilde \gamma}$ of ${\tilde \gamma}(t)$ is transverse to the foliation $\fol$ if $m \neq 0$, and is contained in the leaf through $p$ if $m = 0$.

In fact, as ${\tilde \gamma}(t) = F(\gamma(t),\Phi(t))$, we have by Equation (\ref{equ:vector}) \[ \frac{d {\tilde \gamma}}{dt} = (F_{\Phi(t)})_{*} (\frac{d \gamma}{dt}) - v({\tilde \gamma}(t)) \frac{d\Phi}{dt} \in T_{{\tilde \gamma}(t)}N. \]
Applying $\alpha$ pointwise we get \[\alpha(\frac{d {\tilde \gamma}}{dt}) = (F_{\Phi(t)}^{*} \a) (\frac{d \gamma}{dt}) - \alpha(v) \frac{d\Phi}{dt} = \alpha(\frac{d \gamma}{dt}) - \frac{d\Phi}{dt} = \alpha(\frac{d \gamma}{dt}) - \a (\frac{d \gamma}{dt}) + m = m,  \] so $\frac{d {\tilde \gamma}}{dt}$ is pointwise transverse to or contained in $\mbox{ker } \a$ depending on the value of $m$.

Note that ${\tilde \gamma}$ may have self--intersection and (when $m = 0$) may fail to be an immersion. However, using a local model, we can us a general position argument to
further homotope ${\tilde \gamma}$ (at the price perhaps of increasing the number of components, when ${\tilde \gamma}$ sits on a leaf) to get the curve $c$ that satisfies the conclusions of the Lemma.

\end{proof}

\subsection{Proof of Lemma \ref{lem:beta}} \label{section:beta} \label{section:aux2}

We are now ready to prove the first of the two auxiliary lemmas, i.e. we will prove the following claim.

\begin{claim}
Let $\a$ be a non--degenerate closed 1--form on $N$ and $e\in H^2(N;\Z)$ such that $e\cup [\a]=0$.
There exists a 1--form $\b$ on $N$ such that $\a\wedge \b$ is closed and $[\b\wedge \a]=e\in
H^2(N;\R)$.
\end{claim}

By Lemma \ref{lem1} we can  find an oriented smoothly embedded curve $c$ dual to $e\in H^2(N;\Z)$ such that $\a|_c\equiv 0$.
We denote the components of $c$ by $c_1,\dots,c_m$.
We now consider $S^1\times D^2$ with the coordinates $e^{2\pi it}, x,y$ and
we orient $S^1\times D^2$ by picking the equivalence class of the basis $\{\partial_x,\partial_y,\partial_t\}$.

Using the orientability of the $N$ and of the leaves of the foliation we use a standard argument to show that for $i=1,\dots,m$ we can pick a map
\[ f_i:S^1\times D^2 \to N\]
with the following properties:
\bn
\item $f_i$ is an orientation preserving diffeomorphism onto its image.
\item $f_i$ restricted to $S^1\times 0$ is an orientation preserving diffeomorphism onto $c_i$.
\item $\a((f_i)_*(\partial_t))=0$.
\item $\a((f_i)_*(\partial_x))=0$.
\item There exists an $r_i\in (0,\infty)$ such that $\a((f_i)_*(\partial_y))=r_i$ everywhere.
\en
Note that (3), (4) and (5) are equivalent to $f_i^*(\a)=r_i\cdot dy$.

For $i=1,\dots,m$ we now pick a function $\rho_i:D^2\to \R_{\geq 0}$ such that the closure of the support of $\rho_i$ lies in the interior of $D^2$ and such that \[ \int_{D^2}\rho_i (x,y) dx \wedge dy =\frac{1}{r_i}. \]
We define the following 1--form on $S^1\times D^2$:
\[ \b'_i(t,x,y)=\rho_i(x,y)\cdot dx.\]
Note that
\be \label{equ:closed} d(\b'_i\wedge f_i^*(\a))=d(\b'_i\wedge r_i\cdot  dy)=d(r_i\rho_i(x,y)\cdot dx \wedge dy)=0.\ee
Furthermore for any $z\in S^1$ we have
\be \label{equ:one} \int_{z \times D^2} \b'_i\wedge f_i^*(\a)=\int_{z \times D^2} r_i\rho_i(x,y)\cdot dx \wedge dy=1.\ee
For $i=1,\dots,m$ we now define the following 1--form on $N$:
\[ \b_i(p)=\left\{ \ba{rl} 0,& \mbox{ if }p\in N\sm f_i(S^1\times D^2)\\
(f_i^{-1})^*(\b'_i(q)),& \mbox{ if }p=f_i(q) \mbox{ for some }q\in S^1\times D^2.\ea \right.\]
Furthermore we let $\b=\sum_{i=1}^m \b_i$.
We claim that $\b$ has all the required properties.

First note that $\b$ is $C^\infty$ by our condition on the support of $\rho_i$.
Furthermore it follows immediately from (\ref{equ:closed}) that $\b \wedge \a$ is closed.
Finally we have to show that $\beta \wedge \a$ represents $e$.

In order to show that $\beta \wedge \a$ represents $e$ in $H^2(N;\R)=\hom(H_2(N;\Z),\R)$
it is enough to show that for any embedded oriented surface $S\subset N$, we have
\[ \int_{S} \beta \wedge \a=e([S]).\]
We first note that $e([S])=c\cdot s$.
It is therefore enough to show that
 for any embedded oriented surface $S\subset N$, we have
\[ \int_{S} \beta_i \wedge \a=c_i\cdot S.\]
In fact, given such a surface we can isotope $S$ in such a way that $S$ intersects the curve $c$ `vertically', i.e. we can assume that
\[ f_i(S^1\times D^2)\cap S=\coprod_{j=1}^{k} \eps_j \cdot f_i(z_1\times D^2) \]
for disjoint $z_i$ and $\eps_i\in \{-1,1\}$. We view this equality as an equality of oriented manifolds, where we give
$z_i\times D^2$ the orientation given by the basis $\{\partial_x,\partial_y\}$. In particular $S$ is transverse to $c_i$.
In this case we have
\[ c_i\cdot S=\sum_{j=1}^k \eps_j.\]
On the other hand it follows from (\ref{equ:one}) that
\[  \int_{S} \beta_i \wedge \a=\sum_{j=1}^k  \int_{\eps_j\cdot (z_j\times D^2)} f_i^*(\beta_i) \wedge f_i^*(\a)
=\sum_{j=1}^k  \int_{\eps_j\cdot (z_j\times D^2)} \beta_i' \wedge f_i^*(\a)
=\sum_{j=1}^k \eps_j.\]
This concludes the proof that $\beta$ has all the required properties.
\subsection{Proof of Lemma \ref{lem:nonzero}} \label{section:nonzero} \label{section:aux3}

The following claim is the last missing piece in the proof of Theorem \ref{thm:const}.

\begin{claim}
Let $\a$ be a non--degenerate closed 1--form on $N$.
Given $h\in H^2(N;\R)$ with $h\cup [\a]> 0$ we can find a representative $\Omega$ of $h$ such that
$\Omega\wedge \a> 0$ everywhere.
\end{claim}

We first consider the case that $h$ is represented by an integral class, i.e. by an element in the image of the map $H^2(N;\Z)\to H^2(N;\R)$.
Let $\fol$ be the foliation corresponding to $\a$.

 Using Lemma \ref{lem1} we can pick for each $p\in N$ a curve  $c_p$ transverse to $\fol$ which goes through $p$ and which represents $h$. Since $N$ is orientable we can  pick maps
\[ f_p:S^1\times D^2\to N \]
such that
\bn
\item $f_p$ is an orientation preserving diffeomorphism onto its image
(where we again view $S^1\times D^2$  with the orientation given by $\{\partial_x,\partial_y,\partial_t\}$).
\item $f_p$ restricted to $S^1\times 0$ is an orientation preserving diffeomorphism onto $c_p$.
\item $\a((f_p)_*(\partial_x))=0$.
\item $\a((f_p)_*(\partial_y))=0$.
\item  $\a((f_p)_*(\partial_t))>0$.
\en
Note that (3) and (4) are equivalent to saying that $(f_p)_*(\partial_x)$ and  $(f_p)_*(\partial_y)$ are tangent to the leaves of the foliation $\fol$.
Also note that on $S^1\times D^2$ we have $dx\wedge dy\wedge (f_p)^*(\a)\ne 0$.

By compactness we can find $p_1,\dots,p_k$ such that
\be \label{equ:cover} \bigcup\limits_{j=1}^k f_{p_i}\big(S^1\times \frac{1}{2}D^2\big)=N.\ee
We write $f_i=f_{p_i}, i=1,\dots,k$.
Now we pick a function $\rho:D^2\to \R_{\geq 0}$
such that  the following conditions hold:
\bn
\item $\int_{D^2}\rho =\frac{1}{k}$,
\item $\rho$ is strictly positive on $\frac{1}{2}D^2$, and
\item the closure of the support of $\rho$ lies in the interior of $D^2$.
\en

Let $\Omega'$ be the 2--form on $S^1\times D^2$ given by
\[ \Omega'(z,x,y)=\rho(x,y)dx\wedge dy.\]
Clearly $\Omega'$ is closed and for any $z\in S^1$ we have $\int_{z \times D^2}\Omega'=\frac{1}{k}$.
For $i=1,\dots,k$ we now define the following 2--form on $N$:
\[ \Omega_i(p)=\left\{ \ba{rl} 0,& \mbox{ if }p\in N\sm f_i(S^1\times D^2)\\
(f_i^{-1})^*(\Omega'(q)),& \mbox{ if }p=f_i(q) \mbox{ for some }q\in S^1\times D^2.\ea \right.\]
As in the proof of Lemma \ref{lem:beta} we see that $\Omega_i$ is smooth, $\Omega_i$ is closed and $[\Omega_i]=\frac{1}{k}h\in H^2(N;\R)$.
Now let $\Omega(h)=\sum_{i=1}^k\Omega_i$. Clearly $[\Omega(h)]=h\in H^2(N;\R)$, and it easily follows from (\ref{equ:cover})
and all the other conditions that
$\Omega(h)\wedge \a>0$  everywhere.
\\

We now turn to the general case, i.e. to the case that   $h\in H^2(N;\R)$ is not necessarily integral.

\begin{lemma} \label{lem:integral}
Let $h\in H^2(N;\R)$ with $h\cup [\a]>0$.  Then we can find $m\in \N$, integral $h_1,\dots,h_m$ and $a_1,\dots,a_m \in \R_{\geq 0}$ such that
$h_i\cup [\a]>0$ for all $i$ and such that
 $h=\sum_{i=1}^m a_ih_i$.
\end{lemma}

We first show that Lemma \ref{lem:integral} implies Lemma \ref{lem:nonzero}. Indeed, given
$h\in H^2(N;\R)$ with $h\cup [\a]>0$ we pick integral $h_1,\dots,h_m$ and $a_1,\dots,a_m \in \R_{\geq 0}$ as above. Then we  define $\Omega(h_1),\dots,\Omega(h_m)$ as above. We let
\[ \Omega=\sum_{i=1}^m a_i \Omega(h_i).\]
We see that
\[ \Omega(h)\wedge \a= \sum_{i=1}^m a_i \Omega(h_i)\wedge \a>0\]
 everywhere. This concludes the proof of Lemma \ref{lem:nonzero} assuming Lemma \ref{lem:integral}.

 We now turn to the proof of Lemma \ref{lem:integral}.
It is easy to see that we can  pick a basis $e_1,\dots,e_n$ for $H^1(N;\Q)$ such that $e_i\cup [\a]>0$ for all $i=1,\dots,m$. We use this basis to identify  $H^2(N;\R)$ with $\R^n$.
 We  say that $h\in H^2(N;\R)$ with $h\cup [\a]>0$ has property ($*$) if  there exists $m\in \N$, integral $h_1,\dots,h_m$ and  $a_1,\dots,a_m \in \R_{\geq 0}$ such that
$h_i\cup [\a]>0$ for all $i$ and such that
 $h=\sum_{i=1}^m a_ih_i$. Note that if $h_1,h_2$ have property ($*$), then $h_1+h_2$ also has property ($*$).

Given $m\in \{0,\dots,n\}$ we now say \emph{$P(m)$ holds if
 ($*$)  holds for all $ g=(g_1,\dots,g_n)\in H^2(N;\R)=\R^n$ with $g_1,\dots,g_m\in \Q$.}
Clearly we have to show that $P(0)$ holds.
Note that $P(n)$ holds since any rational element of $H^2(N;\R)$ is a non--negative multiple of an integral element.

We now show that $P(m+1)$ implies that $P(m)$ holds as well.
So assume $P(m+1)$ holds and that we have
\[ g=(g_1,\dots,g_m,g_{m+1},\dots,g_n)\]
with $g_1,\dots,g_m\in \Q$ and $h\cdot [\a]>0$.
By continuity we can find $r>0$ such that $g_{m+1}-r\in \Q$ and with the property
that
\[ (g_1,\dots,g_m,g_{m+1}-r,\dots,g_n)\cdot [\a]>0.\]
We write
\[ (g_1,\dots,g_m,g_{m+1},\dots,g_n)=(g_1,\dots,g_m,g_{m+1}-r,\dots,g_n)+r e_{m+1}.\]
The claim now follows from $P(m+1)$ and $e_{m+1}\cup [\a]>0$.

\section{Proof of Theorem \ref{thm:symplecticcone}}\label{section:3}

We first prove the following proposition.

\begin{proposition} \label{prop:nondeg}
Let $M$ be a $4$--manifold with a free circle action. Denote by $p:M\to N$ the projection map to the
orbit space. Assume that $(N,p_*([\w]))$ fibers over $S^1$ for any symplectic form $\w$ such that  $p_*([\w])$ is an integral class which is  primitive in $H^1(N;\Z)$.
Then for any symplectic form $\w$ the class $p_*([\w])\in
H^1(N;\R)$ can be represented by a non--degenerate closed 1--form.
\end{proposition}

\begin{proof}
First let $\w$ be a  symplectic form such that  $p_*([\w])\in H^1(N;\Q)$.
We can find $s\in \Q$ such that
$sp_*([\w])=p_*([s\w])$ is a primitive element in $H^1(N)$. By assumption $(N,sp_*([\w]))$ fibers
over $S^1$, in particular $sp_*([\w])$ (and hence $p_*([\w]$) can be represented by a
non--degenerate closed 1--form.

Now let $\w$ be a  symplectic form such that $p_*([\w])\in H^1(N;\R)\sm H^1(N;\Q)$, and let $C$ be the open cone over the face of the unit ball of the Thurston norm in which $C$ lies.
 (Note that $C$ is a priori not necessarily top dimensional.)
 Since the vertices of the Thurston norm ball are rational (cf. \cite[Section~2]{Th86}), and by the openness of the symplectic condition, we can find a symplectic form $\w'$ on
$M$ such that $p_*([\w'])$ is in $H^1(N;\Q)$ and is contained in the cone $C$ as well. By the previous observation it follows that there exists at least one element of $C$ (namely $p_*([\w'])$ itself) that can be represented by a
non--degenerate closed 1--form. But then by \cite[Theorem~5]{Th86} all  elements in $C$, in particular  $p_*([\w])$, can be represented by non--degenerate closed $1$--forms.
\end{proof}

We can now prove Theorem \ref{thm:symplecticcone}.

\begin{proof}[Proof of Theorem \ref{thm:symplecticcone}]
Let $N$ be a closed oriented 3--manifold and let $\psi \in H^2(S^1\times N;\R)$. We have to show that  the following are equivalent:
\bn
\item $\psi$ can be represented by a symplectic structure;
\item $\psi$ can be represented by a symplectic structure which is $S^1$--invariant;
\item $\psi^2>0$ and the K\"unneth component $\phi \in H^1(N;\R)$ of $\psi$ lies in the open cone on a fibered face of the Thurston norm ball of $N$.
\en
Clearly (2) implies (1). Theorem \ref{thm:const} shows that (3) implies (2).
By  the results of \cite[Theorems~1.2~and~1.4]{FV08}  we know that
 for any symplectic form $\w$ with $p_*([\w])\in H^1(N)$ primitive the
pair $(N,p_*([\w]))$ fibers over $S^1$.
(Note that this is stated only for integral forms $[\w]$ but the argument in \cite{FV08} carries through
for any $[\w]$ such that $p_*([\w])$ is primitive.)
Proposition \ref{prop:nondeg} then asserts  that for any symplectic form $\w$ the class $p_*([\w])\in H^1(N;\R)$ can be represented by a non--degenerate closed 1--form.
\end{proof}

Theorem \ref{thm:symplecticcone} lets us determine the symplectic cone for a significant class of $4$--manifolds.
Our result suggests that the symplectic cone shares the properties of the fibered cone of a 3--manifold.
We propose the following conjecture.

\begin{conjecture}\label{conj1}
Let $W$ be a symplectic $4$--manifold. Then there exists a (possibly non--compact) polytope $C\subset H^2(W;\R)$ with the following properties:
\bn
\item The dual polytope in $H_2(W;\R)$ is compact, symmetric, convex and integral.
\item There exist open top--dimensional faces $F_1,\dots,F_s$ of $C$ such that the symplectic cone coincides with all non--degenerate elements in the cone on $F_1,\dots,F_s$.
\en
\end{conjecture}



\end{document}